\documentclass[12pt]{amsart}
\usepackage{amssymb,slashed, color,array}
\usepackage{amsmath,amsfonts,amsthm,euscript,mathrsfs,multirow}
\usepackage[all]{xy}
\usepackage[english]{babel}
\usepackage{epsf}
\usepackage{graphicx}
\csname @addtoreset\endcsname{equation}{section}

\newtheorem{theorem}{Theorem}[section]

\newtheorem{corol}[theorem]{Corollary}
\newtheorem{prop}[theorem]{Proposition}

\theoremstyle{definition}
\newtheorem{definition}[theorem]{Definition}
\newtheorem{remark}{Remark}[section]

\newenvironment{notation and conventions}{\textbf{Notation and conventions.}}{ }
\DeclareFontFamily{U}{rsf}{} \DeclareFontShape{U}{rsf}{m}{n}{ <5> <6> rsfs5 <7> <8> <9> rsfs7 <10-> rsfs10}{}
\DeclareMathAlphabet\Scr{U}{rsf}{m}{n}
\definecolor{pink}{rgb}{1,0,1}

  \usepackage[pdftex]{hyperref}

\begin{document}

\begin{center}
\baselineskip=14pt{\LARGE
 On characteristic classes of singular hypersurfaces and involutive symmetries of the Chow group\\
}
\vspace{1.5cm}
{\large  James Fullwood$^{\spadesuit}$
  } \\
\vspace{.6 cm}

${}^\spadesuit$Institute of Mathematical Research, The University of Hong Kong, Pok Fu Lam Road, Hong Kong.\\

\end{center}

\vspace{1cm}
\begin{center}

{\bf Abstract}
\vspace{.3 cm}
\end{center}

{\small
For any algebraic scheme $X$ and every $(n,\mathscr{L})\in \mathbb{Z}\times \text{Pic}(X)$ we define an associated involution of its Chow group $A_*X$, and show that certain characteristic classes of (possibly singular) hypersurfaces in a smooth variety are interchanged via these involutions. For $X=\mathbb{P}^N$ we show that such involutions are induced by involutive correspondences.  
}
 \tableofcontents{}






\section{Introduction}\label{intro}
Fix an algebraically closed field $\mathfrak{K}$ of characteristic zero, let $M$ be a smooth $\mathfrak{K}$-variety and let $X\subset M$ be a hypersurface. For singular $X$ there exists a generalization of the notion of `Milnor number' to arbitrary singularities which is a characteristic class supported on the singular locus of $X$ referred to in the literature as the \emph{Milnor class} of $X$, which we denote by $\mathcal{M}(X)$. The Milnor class of $X$ may be defined (up to sign) as the difference between its \emph{Fulton class} $c_{\text{F}}(X)$ and its \emph{Chern-Schwartz-MacPherson} (or simply \emph{CSM}) class $c_{\text{SM}}(X)$. Both the Fulton class and CSM class are elements of the Chow group $A_*X$ which are generalizations of Chern classes to the realm of singular varieties in the sense that the classes both agree with the total homology Chern class in the case that $X$ is smooth\footnote{We give a more in-depth discussion of all classes mentioned here in \S\ref{S1}. For an introduction to characteristic classes for singular varieties proper we recommend \cite{CharacteristicSingularVariety}.}. Another characteristic class suppoerted on the singular locus of a hypersurface $X$  is the \emph{L\^e-class} of $X$, denoted $\Lambda(X)\in A_*X$, which was first defined in \cite{SeadeLe} and named as such as they are closely related to the so-called \emph{L\^e-cycles}  of $X$, which were initially defined and studied independent of Milnor classes \cite{MasseyLe}. The attractive result of \cite{SeadeLe} is that both $\mathcal{M}(X)$ and $\Lambda(X)$ determine each other in a completely symmetric way, i.e.,

\begin{equation}
\label{mcf}
\mathcal{M}_k(X)=\sum_{j=0}^{d-k}(-1)^{j+k}\left(\begin{array}{c} j+k \\ k \end{array}\right)c_1(\mathscr{O}(X))^j\cap \Lambda_{j+k},
\end{equation}
and

\begin{equation}
\label{lcf}
\Lambda_k(X)=\sum_{j=0}^{d-k}(-1)^{j+k}\left(\begin{array}{c} j+k \\ k \end{array}\right)c_1(\mathscr{O}(X))^j\cap \mathcal{M}(X)_{j+k},
\end{equation}
\\
where $d$ is the dimension of the singular locus of $X$ and a $k$th subscript on a class denotes its component of dimension $k$. 

This fact seems to suggest the existence of some non-trivial involutive symmetry of $A_*X$ which exchanges $\mathcal{M}(X)$ and $\Lambda(X)$, which we show in \S\ref{S3}  is the case. In fact, both $\mathcal{M}(X)$ and $\Lambda(X)$ may be recovered from the relative \emph{Segre class} (see Definition \ref{scd}) $s(X_s,M)$ of the singular scheme $X_s$ of $X$ (i.e., the subscheme of $X$ whose ideal sheaf is locally generated by the partial derivatives of a local defining equation for $X$), and we show that there exists a countable infinity of such involutive symetries of $A_*X$ which exchange $\mathcal{M}(X)$ and classes closely related to $s(X_s,M)$, for which the result of \cite{SeadeLe} is but one of them.

The proofs in \cite{SeadeLe} of formulas \ref{mcf} and \ref{lcf} are topological in nature, span many pages and involve heavy machinery such as derived categories and Whitney stratifications, obfuscating any conceptual insight as to why such formulas should hold. However, in an unpublished note \cite{AluffiLe}, Aluffi shows that the proofs may be reduced to merely a few lines using his `intersection-theoretic calculus' (which is purely algebraic). Our observation is that Aluffi's proof may be immediately generalized to show that for every $(n,\mathscr{L})\in \mathbb{Z}\times \text{Pic}(X)$ (for any algebraic $\mathfrak{K}$-scheme $X$) there corresponds an involutive symmetry of $A_*X$, and that not only Milnor classes and L\^e-cycles, but other characteristic classes such CSM classes and Aluffi classes of singular schemes of hypersurfaces (which may be integrated to yield the Donaldson-Thomas type invariant of the singular scheme) are both interchanged in a similar way with classes closely related to the Segre class of its singular scheme as well. As such, the theme of this note is that the symmetric formulas \ref{mcf} and \ref{lcf} are but a special case of a more general phenomenon of characteristic classes of hypersurfaces, which are induced by involutive symmetries of their Chow groups. 
\\

\noindent \emph{Acknowledgements.} We would like to thank Jos\`{e} Seade for sharing with us the preprint \cite{SeadeLe}, and for our discussions on Milnor classes. We also thank Paolo Aluffi for sharing with us the unpublished note \cite{AluffiLe}, for which this article is little more than a mere rewriting of.

\section{Characteristic classes of singular hypersurfaces}\label{S1} The total Chern class $c(X)$ of a smooth $\mathfrak{K}$-variety $X$ is \emph{the} fundamental characteristic class for $\mathfrak{K}$-varieties in the sense that all other reasonable notions of characteristic class are linear combinations of Chern classes over a suitable ring. For those interested in singularities, it is then only natural that one would want to generalize the notion of Chern class to the realm of singular varieties (and schemes) in such a way that that they agree with the usual Chern class for smooth varieties. The CSM class $c_{\text{SM}}(X)$ of a possibly singular variety $X$ is in some sense the most direct generalization, since for $\mathfrak{K}=\mathbb{C}$ it generalizes the Poincar\'e-Hopf (or Gau{\ss}-Bonnet) theorem to the realm of singular varieties, i.e., 

\[
\int_X c_{\text{SM}}(X)=\chi(X),
\]
\\ 
where $\chi(X)$ denotes the topolocial Euler characteristic with compact supports, and the integral sign is notation for proper pushforward to a point\footnote{We note that while CSM classes were first defined over $\mathbb{C}$ \cite{MCC}, their definition was later generalized to an arbitrary algebraically closed field of characteristic zero in \cite{KCC}.} (whose Chow group is $\mathbb{Z}$). For arbitrary $\mathfrak{K}$ (algebraically closed of characteristic zero) we simply \emph{define} the Euler characteristic of a $\mathfrak{K}$-variety as the `integral' of its CSM class. Moreover, CSM classes are a generalization of counting in the sense that they obey inclusion-exclusion (which of course is very useful for computations). In \cite{MR1697199}, Aluffi obtained a very nice formula for the CSM class of a hypersurface in terms of the Segre class (see Definition \ref{scd}) of its singular scheme, and since we are only concerned with hypersurfaces in this note we may use his formula as a working definition (we recall Aluffi's formula in \S \ref{S3}, after introducing some useful notations). 

Another class generalizing the Chern class to the realm of singular varieties and schemes is the Fulton class, which is defined for any subscheme of a smooth $\mathfrak{K}$-variety $M$\footnote{From here on we will refer to such schemes as \emph{embeddable schemes}.}. For $X$ a (possibly singular) local complete intersection, its Fulton class $c_{\text{F}}(X)$ agrees (after pushforward to $M$) with the total Chern class of a smooth variety in the same rational equivalence class as $X$, and so $c_{\text{SM}}(X)$ differs from $c_{\text{F}}(X)$ only in terms of dimension less than or equal to the dimension of its singular locus. The difference $c_{\text{SM}}(X)-c_{\text{F}}(X)$ then measures the discrepancy of $c_{\text{SM}}(X)$ from the Chern class of a smooth deformation of $X$ (parametrized by $\mathbb{P}^1$), and is an invariant precisely of the \emph{singularities} of $X$. For $X$ with only isolated singularities (over $\mathbb{C}$) the integral of $c_{\text{SM}}(X)-c_{\text{F}}(X)$ agrees (up to sign) precisely with the sum of the Milnor numbers of each singular point of $X$, thus it seemed natural to refer to this class generalization of global Milnor number as the `Milnor class' of $X$, which we denote by $\mathcal{M}(X):=c_{\text{SM}}(X)-c_{\text{F}}(X)$\footnote{We blindly ignore any sign conventions some may associate with this class in the literature.} (recall that we are assuming $X$ is a local complete intersection here, so that we are only defining Milnor classes in this context). 

To define the Fulton class of an arbitrary embeddable scheme, we first need the following

\begin{definition}
\label{scd}
Let $M$ be a smooth $\mathfrak{K}$-variety and $Y\hookrightarrow M$ a subscheme. For $Y$ regularly embedded (so that its normal cone is in fact a vector bundle, which we denote by $N_YM$), the \emph{Segre class} of $Y$ relative to $M$ is denoted $s(Y,M)$, and is defined as

\[
s(Y,M):=c(N_YM)^{-1}\cap [Y]\in A_*Y.
\]
\\ 
For $Y$ `irregularly' embedded, let $f:\widetilde{M}\to M$ be the blowup of $M$ along $Y$ and denote the exceptional divisor of $f$ by $E$. The Segre class of $Y$ relative to $M$ is then defined as

\[
s(Y,M):={\left.f\right|_E}_*s(E,\widetilde{M})\in A_*Y,
\]
\\ 
where ${\left.f\right|_E}_*$ denotes the proper pushforward of $f$ restricted to $E$. As $E$ is always regularly embedded, this is enough to define the Segre class of $Y$ (relative to $M$) in any case. 
\end{definition}

The Fulton class is then given by the following

\begin{definition} 
\label{fcd}
Let $Y$ be a subscheme of some smooth variety $M$. It's \emph{Fulton class} is denoted $c_{\text{F}}(Y)$, and is defined as

\[
c_{\text{F}}(Y):=c(TM)\cap s(Y,M)\in A_*Y.
\]
\end{definition}

\begin{remark} As shown in \cite{IntersectionTheory} (Example 4.2.6), $c_{\text{F}}(Y)$ is intrinsic to $Y$, i.e., it is independent of an embedding into some smooth variety (thus justifying the absence of an ambient $M$ anywhere in its notation). 
\end{remark}

\begin{remark} 
\label{r1}
While the Fulton class is sensitive to scheme structure, it may be shown that the CSM class of a scheme coincides with that of its support with natural reduced structure, and thus is \emph{not} sensitive to any non-trivial scheme structure. As for Milnor classes, since they are defined as the difference between the CSM and Fulton classes, they are scheme theoretic-invariants as well. More precisely, in the case of a possibly singular/non-reduced hypersurface $X$, $\mathcal{M}(X)$ is an invariant of the \emph{singular scheme} of $X$, i.e., the subscheme of $X$ whose ideal sheaf is locally generated by the partial derivatives of a local defining equation for $X$. We note that at present it is not clear what scheme structure on the singular locus of an arbitrary local complete intersection determines its Milnor class, though for a large class of global complete intersections it was shown in \cite{FMC} that the Milnor class is determined by a direct generalization of the notion of singular scheme of a hypersurface to complete intersections.
\end{remark}

\begin{remark} The L\^e-class $\Lambda(X)$ of a hypersurface $X$ mentioned in \S\ref{intro} and appearing in formulas \ref{mcf} and \ref{lcf} is essentially just a modification of the Segre class of its singular scheme $X_s$ relative to $M$, as it may be defined as\footnote{An erratum has been appended to the original version of \cite{SeadeLe} which modifies the definition of $\Lambda(X)$ as their original definition was not one that upheld formulas \ref{mcf} and \ref{lcf} as true statements. In the erratum modifications are not only made to the definition of $\Lambda(X)$, but also to the form of formulas \ref{mcf} and \ref{lcf}.  In any case, the definition given here of $\Lambda(X)$ (or any equivalent formulation) is precisely the only one which yields the original formulas \ref{mcf} and \ref{lcf} appearing in \cite{SeadeLe} as true statements.} 
\[
\Lambda(X):=c(\mathscr{O}(X))c(T^*M\otimes \mathscr{O}(X))\cap s(X_s,M)\in A_*X_s,
\]
where $T^*M$ denotes the cotangent bundle of $M$.
\end{remark}

As noted in Remark \ref{r1}, while Fulton classes are sensitive to scheme structure, they are not sensitive to the singularities of a hypersurface (or more generally a local complete intersection), since (as mentioned earlier) the Fulton class of a local complete intersection coincides with that of a smooth representative of its rational equivalence class. A scheme-theoretic characteristic class which is also sensitive to the singularities of an embeddable scheme $Y$ is the \emph{Aluffi class} of $Y$, denoted by $c_{\text{A}}(Y)$, which may be integrated to yield the Donaldson-Thomas type invariant of $Y$\footnote{Aluffi classes were first defined by Aluffi in \cite{WCM}, where he referred to them as \emph{weighted Chern-Mather classes}. Behrend then later coined the term `Aluffi class' in \cite{DTI}, where he makes the first connection between Aluffi's weighted Chern-Mather classes (albeit with a different sign convention) and Donaldson-Thomas type invariants.}. For $Y$ the singular scheme of of a hypersurface $X$ it was shown in \cite{WCM} that (up to sign) $c_{\text{A}}(Y)=c(\mathscr{O}(X))\cap \mathcal{M}(X)$, and since this is the only context in which we consider Aluffi classes we refer the reader to both \cite{DTI}\cite{WCM} for precise definitions and further discussion. 

\section{The involutions $i_{n,\mathscr{L}}$}\label{S2}
Let $X$ be an algebraic $\mathfrak{K}$-scheme. For every $(n,\mathscr{L})\in \mathbb{Z}\times \text{Pic}(X)$ we now define a map $i_{n,\mathscr{L}}:A_*X\to A_*X$, and show that it is an involutive automorphism of $A_*X$ (these will be precisely the involutions which relate various characteristic classes alluded to above). But before doing so, we first introduce two intersection theoretic operations, which will not only provide an efficient way for defining the involutions $i_{n,\mathscr{L}}$, but will also be of computational utility.

So let $\alpha\in A_*X$ be written as $\alpha=\alpha^0+\cdots +\alpha^n$, where $\alpha^i$ is the component of $\alpha$ of \emph{codimension} $i$ (in $X$). We denote by $\alpha^{\vee}$ the class
\[
\alpha^{\vee}:=\sum (-1)^i \alpha^i,
\]
and refer it it as the `dual' of $\alpha$.

We now define an action of $\text{Pic}(X)$ on $A_*X$. Given a line bundle $\mathscr{L}\to X$  we denote its action on $\alpha=\sum \alpha^i\in A_*X$ by $\alpha \otimes_X \mathscr{L}$, which we define as 

\[
\alpha \otimes_X \mathscr{L}:= \sum \frac{\alpha^i}{c(\mathscr{L})^i}.
\] 
It is straightforward to show that this defines an honest action (i.e., $(\alpha \otimes_X \mathscr{L})\otimes_X \mathscr{M}=\alpha \otimes_X (\mathscr{L}\otimes \mathscr{M})$ for any line bundles $\mathscr{L}$ and $\mathscr{M}$), and we refer to this action as `tensoring by a line bundle'. For $\mathscr{E}$ a rank $r$ class in the Grothendieck group of vector bundles on $X$ (note that $r$ may be non-positive), the formulas

\begin{equation}
\label{df}
\left(c(\mathscr{E})\cap \alpha\right)^{\vee}=c(\mathscr{E}^{\vee})\cap \alpha^{\vee}
\end{equation}
\begin{equation}
\label{tf}
\left(c(\mathscr{E})\cap \alpha\right) \otimes_X \mathscr{L}=\frac{c(\mathscr{E}\otimes \mathscr{L})}{c(\mathscr{L})^r}\cap \left(\alpha \otimes_X \mathscr{L}\right)
\end{equation}
\\ 
were proven in \cite{MR1316973} (along with the first appearance of the `tensor' and `dual' operations), and will be indispensable throughout the remainder of this note\footnote{The tensor and dual operations, along with formulas \ref{df} and \ref{tf} are what we refer to as Aluffi's `intersection-theoretic calculus' in \S\ref{intro}.}.  We now arrive at the following
 
\begin{prop}\label{inv}
Let $X$ be an algebraic $\mathfrak{K}$-scheme, $n\in \mathbb{Z}$ and $\mathscr{L}\to X$ be a line bundle. Then the map $i_{n, \mathscr{L}}:A_*X\to A_*X$ given by

\[
\alpha \mapsto c(\mathscr{L})^n\cap \left(\alpha^{\vee}\otimes_X \mathscr{L} \right)
\]
\\
is an involutive automorphism of $A_*X$ \emph{(}i.e., $i_{n, \mathscr{L}}\circ i_{n, \mathscr{L}}=\emph{id}_{A_*X}$\emph{)}.
\end{prop}

\begin{proof} Let $\alpha\in A_*X$ and denote $i_{n, \mathscr{L}}(\alpha)$ by $\beta$, i.e.,

\begin{equation}
\label{be}
\beta=c(\mathscr{L})^n\cap \left(\alpha^{\vee}\otimes_X \mathscr{L} \right).
\end{equation}
We will show that $i_{n, \mathscr{L}}(\beta)=\alpha$, which implies the conclusion of the proposition. Capping both sides of the equation \ref{be} by $c(\mathscr{L})^{-n}$ we get

\begin{equation}
\label{ce}
c(\mathscr{L})^{-n}\cap \beta=\alpha^{\vee}\otimes_X \mathscr{L}.
\end{equation}
\\
By formula \ref{tf}, for any line bundle $\mathscr{M}\to X$ we have

\[
\left( c(\mathscr{L})^{-n}\cap \beta \right)\otimes_X \mathscr{M}=\frac{c(\mathscr{M})^n}{c(\mathscr{L}\otimes \mathscr{M})^n}\cap (\beta\otimes_X \mathscr{M}),
\]
thus tensoring both sides of equation \ref{ce} by $\mathscr{L}^{\vee}$ yields

\begin{equation}
\label{le}
c(\mathscr{L}^{\vee})^n\cap \left(\beta\otimes_X \mathscr{L}^{\vee}\right)=\alpha^{\vee}.
\end{equation}
\\
Finally, taking the `dual' (i.e. applying formula \ref{df}) to both sides of equation \ref{le} we have

\[
\alpha=c(\mathscr{L})^n\cap \left(\beta^{\vee}\otimes_X \mathscr{L}\right)=i_{n, \mathscr{L}}(\beta),
\]
as desired.

The fact that $i_{n, \mathscr{L}}$ is a homomorphism (i.e. $\mathbb{Z}$-linear) follows from the fact that dualizing, tensoring by a line bundle and capping with Chern classes are all linear operations.
\end{proof}

\begin{remark}
The map $\alpha\mapsto \alpha^{\vee}$ sending a class to its dual coincides with $i_{n,\mathscr{O}}$ for every $n\in \mathbb{Z}$. 
\end{remark}

\section{Symmetric formulas abound}\label{S3} We now assume $M$ is a smooth proper $\mathfrak{K}$-variety and $X\subset M$ is an arbitrary hypersurface\footnote{By `hypersurface' we mean the zero-\emph{scheme} associated with a generic section of line bundle on $M$.} with singular scheme $X_s$. In what follows, as we prefer to work mostly in $M$, we will not distinguish between classes in $A_*X$ and their pushforwards (via the natural inclusion) to $A_*M$. We will call two classes $k$ -$\mathscr{L}$ \emph{dual} if one is the image of the other (and so vice-versa) under the map $i_{k,\mathscr{L}}$. We show formulas \ref{mcf} and \ref{lcf} are consequences of the fact that $\mathcal{M}(X)$ and $\Lambda(X)$ are simply $\text{dim}(M)$-$\mathscr{O}(X)$ dual, and show how other classes mentioned in \S\ref{S1} have similar expressions in terms of `dual partners'. 

Everything here is essentially an application of Proposition \ref{inv} and formulas \ref{df} and \ref{tf} to Aluffi's formula for the CSM class of $X$:

\begin{theorem}[Aluffi, \cite{MR1697199}]

\[
c_{\emph{SM}}(X)=\frac{c(TM)}{c(\mathscr{O}(X))}\cap\left([X]+s(X_s,M)^{\vee}\otimes_M \mathscr{O}(X)\right).
\]
\vspace{0.1in}
\end{theorem}

We then immediately arrive at the following

\begin{corol}
\label{mcf2}
\[
\mathcal{M}(X)=\frac{c(TM)}{c(\mathscr{O}(X))}\cap\left(s(X_s,M)^{\vee}\otimes_M \mathscr{O}(X)\right).
\]
\vspace{0.1in}
\end{corol}

\begin{proof}
This follows directly from definitions of Fulton class and Milnor class, as $\mathcal{M}(X)=c_{\text{SM}}(X)-c_{\text{F}}(X)$ and $c_{\text{F}}(X)=c(TM)\cap s(X,M)=\frac{c(TM)}{c(\mathscr{O}(X))}\cap [X]$. 
\end{proof}

Formulas \ref{mcf} and \ref{lcf} (i.e., the main result of \cite{SeadeLe}) are then a special case of the following

\begin{theorem} 
\label{mcs}
Let $n$ be an integer. Then 

\[
\mathcal{M}(X)=i_{n,\mathscr{O}(X)}(\alpha_X(n)) \quad \text{and} \hspace{0.15in}\alpha_X(n)=i_{n,\mathscr{O}(X)}(\mathcal{M}(X)),
\]
where
\[
\alpha_X(n):=c(T^*M\otimes \mathscr{O}(X))c(\mathscr{O}(X))^{n+1-\emph{dim}(M)}\cap s(X_s,M).
\]
\end{theorem}

\begin{proof}
By Corollary \ref{mcf2} we have

\begin{eqnarray*}
\mathcal{M}(X)&=&\frac{c(TM)}{c(\mathscr{O}(X))}\cap\left(s(X_s,M)^{\vee}\otimes_M \mathscr{O}(X)\right) \\
							&=&c(\mathscr{O}(X))^n\cap \left(\frac{c(TM)c(\mathscr{O})^{n+1-\text{dim}(M)}}{c(\mathscr{O}(X))^{n+1}}\cap \left(s(X_s,M)^{\vee}\otimes_M \mathscr{O}(X)\right)\right) \\
							&\overset{\ref{tf}}=&c(\mathscr{O}(X))^n\cap \left(\left(c(TM\otimes \mathscr{O}(-X))c(\mathscr{O}(-X))^{n+1-\text{dim}(M)}\cap s(X_s,M)^{\vee}\right)\otimes_M \mathscr{O}(X)\right) \\
							&\overset{\ref{df}}=&c(\mathscr{O}(X))^n \cap \left(\left(c(TM^*\otimes \mathscr{O}(X))c(\mathscr{O}(X))^{n+1-\text{dim}(M)}\cap s(X_s,M)\right)^{\vee}\otimes_M \mathscr{O}(X)\right) \\
							&=& i_{n,\mathscr{O}(X)}(\alpha_X(n)).
\end{eqnarray*}
The formula $\alpha_X(n)=i_{n,\mathscr{O}(X)}(\mathcal{M}(X))$ then follows as $i_{n,\mathscr{O}(X)}$ is an involution by Proposition \ref{inv}.
\end{proof}

\begin{remark}
In \cite{SSH}, the class $c(T^*M\otimes \mathscr{O}(X))\cap s(X_s,M)$ was taken as the definition of a class referred to as the $\mu$-\emph{class} of $X_s$ (as it generalized Parusi\`{n}ski's `$\mu$-number' \cite{PMN}), which was denoted $\mu_{\mathscr{O}(X)}(X_s)$. Thus 

\[
\alpha_X(n)=c(\mathscr{O}(X))^{n+1-\text{dim}(M)}\cap \mu_{\mathscr{O}(X)}(X_s) 
\]
and
\[
\mathcal{M}(X)=i_{\text{dim}(X),\mathscr{O}(X)}(\mu_{\mathscr{O}(X)}(X_s))
\]
by Theorem \ref{mcs}. Moreover, it was shown in \cite{SSH} that the $\mu$-class of the singular scheme of a hypersurface is intrinsic to the singular scheme, i.e., if $X_s$ coincided with the singular scheme of another hypersurface embedded in a different smooth variety the class remains unchanged. Applications of $\mu$-classes to the study of dual varieties varieties and contact schemes of hypersurfaces were also considered in \cite{SSH}.   

\end{remark}

\begin{remark}As $n$ varies over $\mathbb{Z}$, writing out the formula for the $k$th dimensional piece $\mathcal{M}_k(X)$ of the Milnor class of $X$ via Theorem \ref{mcs} yields infinitely many symmetric formulas similar to \ref{mcf} and \ref{lcf}. Moreover, $\alpha_X(\text{dim}(M))=\Lambda(X)$ (the L\^{e} class of $X$), which implies formulas \ref{mcf} and \ref{lcf}, as we now show: 
\end{remark}

\begin{corol}[\cite{SeadeLe}]
Formulas \ref{mcf} and \ref{lcf} hold.
\end{corol}

\begin{proof}
Denote the dimension of $X$ by $d$. By Theorem \ref{mcs},

\begin{eqnarray*}
\mathcal{M}(X)&=&i_{d,\mathscr{O}(X)}(\alpha_X(d)) \\
              &=&i_{d,\mathscr{O}(X)}(\Lambda(X)) \\
							&=&c(\mathscr{O}(X))^d\cap \left(\Lambda(X)^{\vee}\otimes_M \mathscr{O}(X)\right) \\
							&=&c(\mathscr{O}(X))^d\cap \left(\sum_{i=0}^{d}\frac{(-1)^i\Lambda_{d-i}(X)}{c(\mathscr{O}(X))^i}\right) \\
							&=&\sum_{i=0}^{d}(-1)^ic(\mathscr{O}(X))^{d-i}\cap \Lambda_{d-i}(X) \\
							&=&\sum_{i=0}^{d}(-1)^i(1+c_1(\mathscr{O}(X)))^{d-i}\cap \Lambda_{d-i}(X) \\
							&=&\sum_{i=0}^d\sum_{j\geq 0}(-1)^i\left(\begin{array}{c} d-i \\ j \end{array}\right)c_1(\mathscr{O}(X))^j\cap \Lambda_{d-i}(X). \\
\end{eqnarray*}

In the last equality the term $c_1(\mathscr{O}(X))^j\cap \Lambda_{d-i}(X)$ is of dimension $d-i-j$, and so $\mathcal{M}_k(X)$ corresponds to setting $i=d-k-j$, which yields

\[
\mathcal{M}_k(X)=\sum_{j\geq 0}(-1)^{d-k-j}\left(\begin{array}{c} j+k \\ j \end{array}\right)c_1(\mathscr{O}(X))^j\cap \Lambda_{j+k}(X),
\]
which is equivalent (up to sign) to formula \ref{mcf} via the identity $\left(\begin{array}{c} a+b \\ a \end{array}\right)=\left(\begin{array}{c} a+b \\ b \end{array}\right)$. The (possible) disparity in sign comes from the fact that in \cite{SeadeLe} their definition of Milnor class differs from ours by a factor of $(-1)^d$. Formula \ref{lcf} then immediately follows as $\mathcal{M}(X)$ and $\Lambda(X)$ are $d$-$\mathscr{O}(X)$ dual.
\end{proof}

\begin{remark}
We note that it was much more work to write out formulas for the individual components $\mathcal{M}_k(X)$ than that of the total Milnor class $\mathcal{M}(X)$ (as in Theorem \ref{mcs}). And this is a general principle when computing characteristic classes, i.e., it is often simpler to compute a \emph{total} class rather than its individual components.
\end{remark} 

\begin{remark}
As mentioned in \S\ref{S1}, in \cite{WCM} Aluffi defined a scheme-theoretic characteristic class for arbitrary embeddable $\mathfrak{K}$-schemes which Behrend refers to as the `Aluffi class' in his theory of Donaldson-Thomas type invariants \cite{DTI}. The analogue of the Gau{\ss}-Bonnet theorem in this theory is the formula

\[
\int_Y c_{\text{A}}(Y)=\chi_{\text{DT}}(Y),
\]
\\  
where $Y$ is an embeddable scheme with Aluffi class $c_{\text{A}}(Y)$, and $\chi_{\text{DT}}(Y)$ denotes the Donaldson-Thomas type invariant of $Y$. If $Y$ is the singular scheme of a hypersurface $X$ it was shown in \cite{WCM} that

\[
c_{\text{A}}(Y)=c(\mathscr{O}(X))\cap \mathcal{M}(X).
\]
\\  
Thus capping both sides of the formulas constituting Theorem \ref{mcs} with $c(\mathscr{O}(X))$ then yields
\end{remark}

\begin{corol} Let $n$ be an integer, $Y$ be the singular scheme of a hypersurface $X$ and let $\alpha_X(n)$ be defined as in Theorem \ref{mcs}. Then

\[
c_{\emph{A}}(Y)=i_{n+1,\mathscr{O}(X)}(\alpha_X(n)) \quad \text{and} \hspace{0.15in} \alpha_X(n)=i_{n+1,\mathscr{O}(X)}(c_{\text{A}}(Y)).
\]
\end{corol}

We conclude this section by identifying the `$n$-$\mathscr{O}(X)$ dual partners' of the CSM class of $X$:

\begin{theorem}
Let $n$ be an integer. Then 

\[
c_{\emph{SM}}(X)=i_{n,\mathscr{O}(X)}(\nu_X(n)+\alpha_X(n)) \quad \text{and} \hspace{0.15in}\nu_X(n)+\alpha_X(n)=i_{n,\mathscr{O}(X)}(c_{\emph{SM}}(X)),
\]
\\
where

\[
\nu_X(n)=c(T^*M\otimes \mathscr{O}(X))c(\mathscr{O}(X))^{n-\emph{dim}(M)}\cap -[X]
\]
\\
and $\alpha_X(n)$ is as defined in Theorem \ref{mcs}.
\end{theorem}

\begin{proof}
By Proposition \ref{inv} and Theorem \ref{mcs}, the proof amounts to showing $c_{\text{F}}(X)=i_{n,\mathscr{O}(X)}(\nu_X(n))$, as $c_{\text{SM}}(X)=c_{\text{F}}(X)+\mathcal{M}(X)$. Thus

\begin{eqnarray*}
c_{\text{F}}(X)&=&c(TM)\cap s(X,M) \\
               &=&c(TM)\cap \left(c(N_XM)^{-1}\cap [X]\right) \\
               &=&c(TM)\cap \left([X]\otimes_M \mathscr{O}(X)\right) \\
							 &=&c(\mathscr{O}(X))^n\cap \left(\frac{c(TM)c(\mathscr{O})^{n-\text{dim}(M)}}{c(\mathscr{O}(X))^n}\cap \left([X]\otimes_M \mathscr{O}(X)\right)\right) \\
							 &\overset{\ref{tf}}=&c(\mathscr{O}(X))^n\cap \left(\left(c(TM\otimes \mathscr{O}(-X))c(\mathscr{O}(-X))^{n-\text{dim}(M)}\cap [X]\right)\otimes_M \mathscr{O}(X)\right) \\
							 &\overset{\ref{df}}=&c(\mathscr{O}(X))^n\cap \left(\left(c(T^*M\otimes \mathscr{O}(X))c(\mathscr{O}(X))^{n-\text{dim}(M)}\cap -[X]\right)^{\vee}\otimes_M \mathscr{O}(X)\right) \\
							 &=&i_{n,\mathscr{O}(X)}(\nu_X(n)),
\end{eqnarray*}
as desired.
\end{proof}

\section{$i_{n,\mathscr{L}}$ via involutive correspondences}\label{S4}
Let $M$ and $N$ be smooth proper $\mathfrak{K}$-varieties. A \emph{correspondence} from $M$ to $N$ is a class $\alpha\in A_*(M\times N)$, and such an $\alpha$ induces homomorphisms $\alpha_*\in \text{Hom}(A_*M,A_*N)$ and $\alpha^*\in \text{Hom}(A_*N,A_*M)$ given by 

\[
\beta\overset{\alpha_*}\longmapsto q_*(\alpha \cdot p^*\beta), \quad \gamma\overset{\alpha^*}\longmapsto p_*(\alpha \cdot q^*\gamma),
\] 
\\
where $p$ is the projection $M\times N\to M$, $q$ is the projection $M\times N\to N$ and `$\cdot$' denotes the intersection product in $A_*(M\times N)$ (which is well defined via the smoothness assumption on $M$and $N$). Correspondences are at the heart of Grothendieck's theory of motives, and generalize algebraic morphisms in the sense that we think of an arbitrary class $\alpha\in A_*(M\times N)$ as a generalization of the graph $\Gamma_f$ of a (proper) morphism $f\in \text{Hom}(M,N)$. Just as a morphism $f\in \text{Hom}(M,N)$ induces morphisms on the corresponding Chow groups via proper pushforward ($f_*$) and flat pullback ($f^*$), the morphisms $\alpha_*$ and $\alpha^*$ are direct generalizations of proper pushforward and flat pullback as $f_*=(\Gamma_f)_*$ and $f^*=(\Gamma_f)^*$. Moreover, correspondences may be composed in such a way that the functorial properties of porper pushforward and flat pullback still hold, i.e., $(\alpha \circ \vartheta)_*=\alpha_*\circ \vartheta_*$ and $(\alpha \circ \vartheta)^*=\vartheta^*\circ \alpha^*$ for composable correspondences $\alpha$ and $\vartheta$. From this perspective we were naturally led to the question of whether or not for an algebraic scheme $X$ the involutions $i_{n,\mathscr{L}}$ defined in \S\ref{S2} are induced by involutive correspondences in $A_*(X\times X)$. We answer this question for $X=\mathbb{P}^N$ via the following 

\begin{theorem}\label{c}
Let $N$ be a positive integer and $(n,m)\in \mathbb{Z}\times \mathbb{Z}$. Then there exists a unique $\alpha=\sum_{i+j\leq N}a_{i,j}x^iy^j\in \mathbb{Z}[x,y]/(x^{N+1},y^{N+1})\cong A_*(\mathbb{P}^N\times \mathbb{P}^N)$ such that $i_{n,\mathscr{O}(m)}=\alpha_*$\footnote{Note that $\alpha_*=\alpha^*$ in this case.}, and the coefficients of $\alpha$ are given by

\[
a_{N-j,i}=(-1)^j\left(\begin{array}{c} n-j \\ i-j \end{array}\right)m^{i-j}.
\]
\\
\end{theorem}

\begin{proof}
Everything follows by direct calculation, the details of which we only sketch. Consider $\mathbb{P}^N\times \mathbb{P}^N$ with the natural projections onto its first and second factors, which we denote by $p$ and $q$ respectively. Denote by $x$ the hyperplane class in the first factor and by $y$ the hyperplane class in the second factor (we use the same notations for their pullbacks via the natural projections). Let $\beta=\sum_{i=0}^N\beta_ix^i\in A_*\mathbb{P}^N$. It follows directly from the definition of $i_{n,\mathscr{O}(m)}$ and induction that

\[
i_{n,\mathscr{O}(m)}(\beta)=\sum_{i=0}^N\left(\sum_{j=0}^N (-1)^j\left(\begin{array}{c} n-j \\ i-j \end{array}\right)m^{i-j}\beta_j\right)y^i.
\]
\\
We now let $\alpha=\sum_{i+j\leq N}a_{i,j}x^iy^j\in A_*(\mathbb{P}^N\times \mathbb{P}^N)$ be arbitrary, compute $\alpha_*(\beta)=q_*(\alpha\cdot p^*\beta)$, set its coefficients equal to those of $i_{n,\mathscr{O}(m)}(\beta)$, and then observe that this determines the $a_{i,j}$ uniquely. Since we are not using a notational distinction for $x$ and its pullback $p^*x$, $p^*\beta$ retains exactly the same form in its expansion with respect to $x$. Now $\alpha\cdot p^*\beta$ is just usual multiplication in the ring $\mathbb{Z}[x,y]/(x^{N+1},y^{N+1})$, and $q_*(\alpha\cdot p^*\beta)$ is just the coefficient of $x^N$ in the expansion of $\alpha\cdot p^*\beta$ with respect to $x$, which yields

\[
\alpha_*(\beta)=\sum_{i=0}^N\left(\sum_{j=0}^Na_{N-j,i}\beta_j\right)y^i. 
\]
\\
By setting $\alpha_*(\beta)=i_{n,\mathscr{O}(m)}(\beta)$ the $a_{i,j}$ are then uniquely determined to be as stated in the conclusion of the theorem.

To see that $q_*(\gamma)$ for arbitrary $\gamma \in A_*(\mathbb{P}^N\times \mathbb{P}^N)$ is indeed the coefficient of $x^N$ in the expansion of $\gamma$ with respect to $x$, one may first view $q$ as the natural projection of the projective bundle $\mathbb{P}(\mathscr{E})$ with $\mathscr{E}$ the trivial rank $N+1$ bundle over $\mathbb{P}^N$ and $\mathscr{O}_{\mathbb{P}(\mathscr{E})}(1)=x$. Then by the projection formula, to compute $q_*(\gamma)$ we need only to compute $q_*(x^i)$ in the expansion of $\gamma$ with respect to $x$, which we do using the notion of \emph{Segre class} of a vector bundle\footnote{We note that the notion of Segre class of a vector bundle is different than the \emph{relative} Segre class we define in \S\ref{S1}. The definition of Segre class may be found in \cite{IntersectionTheory}, where Fulton first defines the Segre class of a vector bundle and then defines the Chern class of a vector bundle as the multiplicative inverse of its Segre class.}. By definition of the \emph{Segre class} of $\mathscr{E}$, denoted $s(\mathscr{E})$, we have

\[
s(\mathscr{E}):=q_*(1+x+x^2+\cdots). 
\]
\\ 
And since $s(\mathscr{E})=c(\mathscr{E})^{-1}=1$, matching terms of like dimension we see that all powers of $x$ map to $0$ except for $x^N$ which maps to $1$. 
\end{proof}

It would be interesting to determine sheaves on $\mathbb{P}^N\times \mathbb{P}^N$ (depending on $n$ and $m$) whose Chern characters coincide with $\alpha$ as given in Theorem \ref{c}. And certainly there must be a larger class of varieties (other than projective spaces) for which an analogue of Theorem \ref{c} holds.

\bibliographystyle{plain}
\bibliography{chowduality}

\end{document}